\documentclass[a4paper,12pt,reqno]{amsart}
\usepackage{amsmath,amsfonts,amssymb,amsthm,enumerate,multicol}%hyperref,
\usepackage{tikz,graphicx}
\usepackage{hyperref}
\usepackage{caption,enumitem,wasysym}
\usepackage{cleveref}
\usepackage{epstopdf}
\usepackage{float,wrapfig}
\usepackage[labelsep=space]{caption}
\usepackage[font=footnotesize]{caption}
\usepackage{xcolor}
\newtheorem{definition}{Definition}[section]
\newtheorem{theorem}{Theorem}[section]

\newtheorem{lemma}{Lemma}[section]

\newtheorem{remark}{Remark}[section]

\allowdisplaybreaks

 \theoremstyle{plain}
\newtheorem{thm}{Theorem}[section]

\newtheorem{prop}[thm]{Proposition}

\usepackage{mathtools}

\usepackage{amsmath,amsfonts,amssymb,amsthm,enumerate,multicol}%hyperref,
\usepackage{tikz}
\usepackage{float}
\usepackage{autobreak}

\allowdisplaybreaks

\numberwithin{equation}{section}

\begin{document}

\begin{center}
\Large{\textbf{Global existence to the discrete Safronov–Dubovski\v{i} coagulation equations and failure of mass-conservation }}
\end{center}
\medskip
%\centerline{by}
\medskip
\centerline{${\text{Mashkoor~Ali}}$ and   ${\text{Ankik ~ Kumar ~Giri$^*$}}$ }\let\thefootnote\relax\footnotetext{$^*$Corresponding author. Tel +91-1332-284818 (O);  Fax: +91-1332-273560  \newline{\it{${}$ \hspace{.3cm} Email address: }}ankik.giri@ma.iitr.ac.in}
\medskip
{\footnotesize
 %please put the address of the second  and third author

  \centerline{ ${}^{}$Department of Mathematics, Indian Institute of Technology Roorkee,}
   %\centerline{Other lines}
   \centerline{Roorkee-247667, Uttarakhand, India}

}

\bigskip

\begin{quote}
{\small {\em \bf Abstract.} This paper presents the existence of global solutions to the discrete Safronov-Dubvoski\v{i} coagulation equations for a large class of coagulation kernels satisfying $\Lambda_{i,j} = \theta_i \theta_j + \kappa_{i,j}$ with $\kappa_{i,j} \leq A\theta_i \theta_j, \ \ \forall \ \ i,j\ge 1$ where the sequence $(\theta_i)_{i\geq 1}$ grows linearly or superlinearly with respect to $i$. Moreover, the failure of mass-conservation of the solution is also addressed which confirms the occurrence of the gelation phenomenon.}
\end{quote}

%\begin{abstract}
 %In this article, existence and uniqueness of global classical solution to the discrete coagulation equation with collisional breakage  are considered for collisional kernel having linear growth. Moreover, mass conservation property and propogation of moments of the solution are also discussed.
%\end{abstract}

\vspace{.3cm}



%\noindent
%{\bf Keywords:} Coagulation; Oort-Hulst-Safronov model; Smoluchowski model; Weak compactness; Existence; %Mass-conservation.\\
%{\rm \bf MSC (2010).} Primary: 45K05, 45G99, 45G10, Secondary: 34K30.\\	

%%%%%%%%%%%%%%%%%%%%%%%%%%%%%%%%%%%%
%%%%%%%%%%%%%%%%%%%%%%%%%%%%%%%%%%%%

%%%%%%%%%%%%%%%%%%%%%%%%%%%%%%%%%%%%
%%%%%%%%%%%%%%%%%%%%%%%%%%%%%%%%%%%%
 \section{\textbf{Introduction}}
 Coagulation is one of the mechanisms that cause cluster expansion, in which clusters (or particles) grow in size through successive mergers. Some areas where coagulation phenomena play a significant role include aerosol research, raindrop creation, astrophysics (formation of planets and galaxies), and animal herding; for example, see \cite{RLD 1972, FSK 2000, GL 1995} . Smoluchowski, a Polish physicist, made one of the first contributions in that direction when he developed an infinite system of ordinary differential equations, now known as the discrete Smoluchowski coagulation equation, to describe the time-evolution of a system of particle clusters that, due to Brownian motion, can get close enough to one another for binary coagulation of clusters to take place, see \cite{SMOL 1916, SMOL 1917}.
 
 %The discrete Smoluchowski coagulation equation depicts the change in the concentration $\omega_i(t)$, $ i\geq 1$ of clusters of size $i$ (or $i$-mers) at time $t \geq  0$  read as
 
 The discrete Smoluchowski coagulation equation depicts the change in concentration of clusters of size $i$ (or $i$-mers) at time $t \geq 0$, written as
 
\begin{align}
\frac{d\omega_i}{dt}&= \frac{1}{2}\sum_{j=1}^{i-1} \Lambda_{j,i-j} \omega_j \omega_{i-j} -\sum_{j=1}^{\infty} \Lambda_{i,j} \omega_i \omega_j \label{SCE}\\
\omega_i(0) &= \omega_i^{in}\geq 0 \label{SCEIC}
\end{align}
for $i \geq 1$. The coagulation kernel $\Lambda_{i,j}$ is a nonnegative and symmetric function i.e.\  $ 0\leq \Lambda_{i,j} = \Lambda_{j,i} \ \ \forall \ \ i, j \ge 1$ that denotes the rate at which clusters of size $i$ coalesce with the clusters of size $j$ to create the larger ones.
%The first term on the right-hand side of \eqref {SCE} accounts for the formation of $i$-clusters by binary coalescence of smaller ones, while the second term accounts for their depletion through coagulation with other clusters. 
%\textcolor{blue}{In \cite{SMOL 1916, SMOL 1917}, Smoluchowski initially introduced a system of mathematical equations of the form \eqref{SCE}--\eqref{SCEIC} which is later referred to as \emph{Smoluchowski coagulation equation}  describing the coagulation of colloids moving in a Brownian motion}.

The first term on the right-hand side of \eqref{SCE} indicates that the $i$-clusters are formed through binary coalescence of smaller ones, while according to the second term, they are depleted through coagulation with other clusters. In \cite{PBD 1999}, Dubovski\v{i} investigated a dispersed system and proposed the Safronov-Dubovski\v{i} coagulation model, in which only binary collisions between particles can occur simultaneously, the mass of each particle is assumed to be proportional to some $m_0>0$, and the particle is the smallest in the system. As a result, any particles with mass $im_0$ are referred to as $i$-mers. When two particles having masses of $im_0 $ and $jm_0$ collide, then particles grow in the system.  If $j \leq i$, a collision between a $i$-mer and a $j$-mer causes the $j$-mer to split into $j$ monomers. Hence, we have another description of the coagulation process that leads to the balanced equation, also known as the discrete Safronov-Dubovski\v{i} equation:
%Collisions between particles with masses $im_0$ and $jm_0$ cause particles in the system to grow. Particles with mass $im_0$ are referred to as $i$-mers, while the smallest particle in the system has mass $m_0$.
%Let $\omega_i(t)$ denotes the the concentration of $i$-clusters at time $t \geq 0$, then the discrete Safronov–Dubovski\v{i} coagulation (DSDC) equations equations are as follows:}
 %Denoting by $\omega_i(t)$ the concentration of $i$-clusters at time $t \geq 0$, the discrete  Safronov–Dubovski\v{i} coagulation (DSDC) equations equations read
\begin{align}
\frac{d\omega_i(t)}{dt}&= \omega_{i-1}(t) \sum_{j=1}^{i-1} j \Lambda_{i-1,j} \omega_j(t)-\omega_i(t)  \sum_{j=1}^{i} j \Lambda_{i,j} \omega_j(t)-\sum_{j=i}^{\infty} \Lambda_{i,j} \omega_i(t) \omega_j(t), \hspace{.2cm} i \in \mathbb{N},\label{DOHSE}\\
\omega_i(0) &= \omega_i^{in} \geq 0, \hspace{.2cm} i \in \mathbb{N}. \label{IC}
\end{align}

The initial value problem \eqref{DOHSE}--\eqref{IC} represents the dynamics of evolution of the concentration $\omega_i(t), i\geq 1$ of clusters of size $i$ at time $t \geq 0$. The rate at which $i$-mers collide with $j$-mers is determined by the coagulation kernel, $\Lambda_{i,j}$ (with $i\neq j$). The first sum of the equation \eqref{DOHSE} describes the introduction of the $i$-mer into the system due to collisions between the $(i-1)$-mers and monomers formed from fragmented $j$-mers. If $i=1$, the initial sum is 0. The second sum represents the loss (or decay) of $i$-mers as a result of monomer merging. It is shown that the collision involves exactly $j$ monomers by multiplying the first and second sums by $j$. The third sum depicts the fragmentation caused by encounters with larger particles, which causes $i$-mers to decay.

%Note that equation \eqref{DOHSE}-\eqref{IC} is a nonlinear initial value problem that describes the dynamics of evolution of the concentration $\omega_i(t)$, $ i \in \mathbb{N}/\{0\}$ of clusters of size $i$ at time $t\ge 0$. The coagulation kernel, $\Lambda_{i,j}$ (with $i \neq j$), specifies the rate at which $i$-mers collide with $j$-mers. The first sum in \eqref{DOHSE} describes the $i$-mer introduction into the system as a result of collisions between $(i-1)$-mers and monomers produced from fragmented $j$-mers. If $i = 1$, the initial sum is zero. The second sum represents the loss (or decay) of $i$-mers due to monomer merging. The first and second sums are multiplied by $j$ to demonstrate that the collision involves exactly $j$ monomers. The third sum represents the decay of $i$-mers due to fragmentation caused by collisions with bigger particles.

\par

The moments of the concentration $\omega= (\omega_i(t))_{i\geq 1}$ of order $m \geq 0$ are then defined as 
\begin{align}
  M_m(\omega(t)):= M_m(t) = \sum_{i=1}^{\infty} i^m \omega_i(t),
\end{align}
where the zeroth ($m=0$) and the first ($m=1$) moments denote the total number and mass of particles in the system, respectively.

%Next, we define the moments of the concentration $\omega= (\omega_i(t))_{i\geq 1}$ of order $m\geq 0$ as
%\begin{align}
 % M_m(\omega(t)):= M_m(t) = \sum_{i=1}^{\infty} i^m \omega_i(t),
%\end{align}
%where the zeroth ($m=0$) and the first ($m=1$) moments denote the total number of particles and total mass of particles in the system, respectively. 
%Observe that, since particles are neither created nor destroyed in the reactions described by \eqref{SCE} and \eqref{DOHSE}, the total mass is expected to be conserved through the time evolution.

Regarding the existence, uniqueness, and mass conservation property of the solutions to the equations \eqref{DOHSE}—\eqref{IC}, various results are known. When the coagulation kernels satisfy $\lim_{j \to \infty} \frac{\Lambda_{i, j}}{j}=0, \hspace{.2cm} i, j \geq 1$ and $\omega^{in} \in l_1$, the existence and uniqueness of weak solutions to the equations \eqref{DOHSE}--\eqref{IC} have been explored in \cite{BAG 2005}. Additionally, a relationship between \eqref{DOHSE} and its continuous version has been established by an appropriate sequence of solutions to  \eqref{DOHSE}. In \cite{DAV 2014}, it was shown that the classical solutions exist globally for the bounded kernel  $j\Lambda_{i, j} \leq M, \hspace{.1cm} j \leq i$ and for unbounded kernel $\Lambda_{i, j} \leq C_{\nu} h_i h_j$ with an additional assumption $\frac{h_i}{i }\to 0$. Moreover,  mass conservation property for $\Lambda_{i, j} \leq C_{\nu} h_i h_j$ for $h_i \leq i^{\frac{1}{2}} $ and uniqueness for bounded kernels, i.e., $\Lambda_{i, j}\leq C_{\nu}, \hspace{.1cm} i, j\geq 1$ have been established. Recently, in \cite{DAS 2021}, existence of global classical solutions  and mass-conservation properties of solutions have been established for $\Lambda_{i,j} \leq (1+i + j)^{\vartheta}$ where $\vartheta \in [0,1]$, whereas uniqueness of solutions is shown under the condition that $\Lambda_{i,j}\leq Ci^{\vartheta}$ for $\vartheta\leq 2$.  Additionally, for linear coagulation kernels and kernels fulfilling $\min\{i,j\}\Lambda_{i,j}\leq (i+j)$ for every $i,j \geq 1$, the existence of weak solutions has been demonstrated in \cite{SRD 2022} and \cite{SR 2022} respectively, whereas the uniqueness result is identical to \cite{DAS 2021}. The results of \cite{DAS 2021} have recently been improved, and various additional interesting results have been established in \cite{MPA 2022}.\par

 Noting that the reactions given by \eqref{SCE} and \eqref{DOHSE} do not result in the creation or destruction of particles, hence the total mass is predicted to remain constant during the course of the time evolution. This, however, is not always the case. In \cite{HEZ 1983, JI 1998, PL 1999, LT 1982}, it is observed that for the discrete smoluchowski coagulation equations \eqref{SCE}--\eqref{SCEIC}, the mass conservation property fails for coagulation kernels of the form $\gamma_{i, j} \geq (ij)^{\frac{\alpha_0}{2}}$, where $\alpha_0 \in (1,2] $. It is worth noting that gelation also occurs in the DSDC equations, and we get the same result as in the discrete Smoluchowski coagulation equation. Hence, we define the gelation time $T_{\text{gel}} \in [0,+\infty]$ of a solution $\omega= (\omega_i)_{i\geq 1}$ to the DSDC equation \eqref{DOHSE}--\eqref{IC} by
\begin{align*}
T_{\text{gel}} = \inf \Big\{ t \geq 0 \Big| \sum_{i=1}^{\infty}i \omega_i(t) < \sum_{i=1}^{\infty} i \omega_i^{in} \Big\}. 
\end{align*}

In this article, we look into the existence of global solutions to the DSDC equations \eqref{DOHSE}--\eqref{IC} for a class of coagulation rates of the following form.
 \begin{align*}
 \Lambda_{i,j} = \theta_i \theta_j + \kappa_{i,j}, \hspace{.2cm}i,j \geq 1,
 \end{align*}
 with $\kappa_{i,j} \leq A\theta_i \theta_j(i,j\geq 1)$ for some $A>0$. In this case, $(\theta_i)$ is a sequence of non-negative real numbers that may increase sub-linearly, linearly, or super-linearly with respect to $i$ (for a more detailed explanation, see assumptions \eqref{COAGRATE}--\eqref{COAGRATE2} below).

 % Here, $(\theta_i)$ is a sequence of non-negative real numbers which may grow either sub-linearly, linearly or super-linearly with respect to $i$ (see assumptions (2.1)—(2.3) below for a more complete statement). 
%In this work, our goal is to establish the existence of global solutions to the equation \eqref{DOHSE}--\eqref{IC} for initial data with  $\sum_{i=1}^{\infty}\omega_i^{in} <\infty$, whereas $(\theta_i)$ grows at least linearly, i.e. $ \theta_i \geq Bi$ for some $B>0$. 

Our goal in this work is to prove the existence of global solutions to \eqref{DOHSE}--\eqref {IC} for initial data with $\sum_{i=1}^{\infty}\omega_i^{in}< \infty$ and $(\theta_i)$ growing at least linearly, i.e., $\theta_i\geq Bi$ for some $B > 0$. The paper is organized as follows. Section \ref{MR} definition of solutions and assumptions made on coagulation coefficients and provides some preliminary results. Proofs of Theorems \ref{TH1} and \ref{TH2} are discussed in sections \ref{PTH1} and \ref{PTH2}, respectively. Finally, we have demonstrated the gelation phenomenon in Section \ref{GP}.

%This paper's major contribution is to demonstrate the existence of global solutions to the equations \eqref{DOHSE}--\eqref{IC} for initial data with $\sum_{i=1}^{\infty}\omega_i^{in} <\infty$, and  $(\theta_i)$ grows at least linearly, i.e. $ \theta_i \geq Bi$ for some $B>0$.

 \section{\textbf{Main Results and Preliminaries}} \label{MR}
 %We first specify what we mean by a solution to \eqref{DOHSE}--\eqref{IC}.
To begin with, we define what we mean by a solution to \eqref{DOHSE}--\eqref{IC}.
 \begin{definition} \label{DEF1}
Let $ T\in (0,+\infty]$ and let $ \omega^{in} =( \omega_i^{in})_{i\geq 1}$ be a sequence of non-negative real numbers. A solution $ \omega=(\omega_i)_{i \geq 1}$ to \eqref{DOHSE}--\eqref{IC} on $[0,T)$ is a sequence of non-negative continuous functions satisfying, for each $i\geq 1$ and $t\in(0,T),$
 \begin{enumerate}
 \item $\omega \in \mathcal{C}([0,T))$ and $\sum_{j=1}^{\infty} \Lambda_{i,j} \omega_j \in L^1(0,t),$
 \item and further 
 \begin{align*}
 \omega_i(t)= \omega_i^{in} + \int_0^t \Big( \omega_{i-1}(s) \sum_{j=1}^{i-1} j \Lambda_{i-1,j} \omega_j(s)-\omega_i(s)  \sum_{j=1}^{i} j \Lambda_{i,j} \omega_j(s)-\sum_{j=i}^{\infty} \Lambda_{i,j} \omega_i(s) \omega_j(s)\Big) ds
 \end{align*}
\end{enumerate}  
 \end{definition}
  %We assume that the coagulation rates $(\Lambda_{i,j})$ satisfies the following assumptions that is there are two sequences of non-negative real numbers $(\theta_i)$ and $(\kappa_{i,j})$ such that
 %Concerning the coagulation rates $(\Lambda_{i,j})$, we assume throughout the paper that there are two sequences of non-negative real numbers $(\theta_i)$ and $(\kappa_{i,j})$ such that 
The coagulation rates $(\Lambda_{i,j})$ are assumed to satisfy the following assumptions, namely there exist two sequences of non-negative real numbers, $(\theta_i)$ and $(\kappa_{i,j})$, such that
 \begin{align}
 \Lambda_{i,j} = \theta_i \theta_j + \kappa_{i,j}. \label{COAGRATE}
 \end{align}
 In addition, we made the following set of assumptions on coagulation rates, which will be used in our analysis, namely

% Two different sets of assumptions on the coagulation rates will be used in the sequel, namely either
 \begin{align}
 \lim_{i \to \infty}\frac{\theta_i}{i}=0 \hspace{.2cm} \text{and} \hspace{.2cm} \lim_{j\to \infty} \frac{\kappa_{i,j}}{j} =0 \hspace{.2cm} \text{for each} i \geq 1,\label{COAGRATE1} 
 \end{align}
 or 
 \begin{align}
 \inf_{i\geq 1} \frac{\theta_i}{i} =B >0 \hspace{.2cm} \text{and} \hspace{.2cm} \kappa_{i,j} \leq A\theta_i \theta_j \hspace{.2cm} \text{for each} i,j \geq 1(A \geq 0).\label{COAGRATE2}
 \end{align}

%We finally define the Banach spaces $Y_{p}$, with $p\geq 0$, by
Finally, we define the Banach spaces $Y_{p}$ with $p\geq 0$ by
\begin{align*}
Y_{p} = \Big\{ y= (y_i)_{i\geq 1}, \sum_{i=1}^{\infty}i^{p} |y_i| <\infty \Big\},
\end{align*}
equipped with the norm
\begin{align*}
\|y\|_{p} = \sum_{i=1}^{\infty} i^{p} |y_i|.
\end{align*}
Due to the physical significance of the DSDC equations, we will only take into account non-negative solutions, i.e., those that remain within the non-negative cone  $Y_{p}^+$ of  $Y_p$, that is,
%We also denote by $Y_{p}^+$ the positive cone of $Y_{p}$, i.e.,
%\textcolor{blue}{Furthermore, we represent the positive cone of $Y_p$ by $Y_{p}^+$, that is,}
\begin{align*}
Y_{p}^+ = \{y \in Y_{p},\hspace{.1cm} y_i \geq 0 \hspace{.15cm} \text{for each}\hspace{.15cm} i \geq 1\}.
\end{align*}

%Our first result extends \cite[Proposition 4]{BAG 2005} to the more general class of coagulation rates described by \eqref{COAGRATE1}--\eqref{COAGRATE2} and reads
Our first result extends \cite[Proposition 4]{BAG 2005} to the wider class of coagulation rates represented by \eqref{COAGRATE1}--\eqref{COAGRATE2}, and it is as follows:

\begin{theorem} \label{TH1}
Assume that the coagulation kernel satisfies \eqref{COAGRATE} and that either \eqref{COAGRATE1} or \eqref{COAGRATE2} holds. For every $\omega^{in} = (\omega_i^{in})_{i\geq 1} \in Y_1^+$, there exists at least one solution $\omega$ to \eqref{DOHSE}--\eqref{IC} on $[0,+\infty)$, such that $\omega(t) \in Y_{1}^+$ for each $t \in[0.+\infty)$, and
\begin{align}
\|\omega(t)\|_1\leq \|\omega^{in}\|_1. \label{AMC}
\end{align}
\end{theorem}
%Our next result is a generalization of Theorem \ref{TH1} to a wider class of initial data, when the coagulation rates satisfy \eqref{COAGRATE2}.
%\textcolor{blue}{The following result is a generalization of the Theorem \ref{TH1} to a broader class of initial data where the coagulation rate satisfies  \eqref {COAGRATE2}.}
The following result generalizes Theorem \ref{TH1} to a broader class of initial data, when the coagulation rate satisfies \eqref{COAGRATE2}.

\begin{theorem} \label{TH2}
Assume that the coagulation kernel satisfies \eqref{COAGRATE} and \eqref{COAGRATE2}, and that $\omega^{in} = (\omega_i^{in})_{i\geq 1}\in Y_0^+$. Then there exists at least one solution $\omega$ to \eqref{DOHSE}--\eqref{IC} on $[0,+\infty)$, such that $\omega(t)\in  Y_0^+$ for each $ t\in[0, +\infty)$ and, for each $t > 0$,
\begin{align}
\|\omega(t)\|_1<\infty. \label{FM}
\end{align}
\end{theorem}
\begin{remark}
%Unfortunately the coagulation kernel $\Lambda_{i,j} = i+j$ does not inculded in both the classes\eqref{COAGRATE1} and \eqref{COAGRATE2} which we have considered. Although the existence result for this kernel is available in literature \cite{MPA 2022} and \cite{SRD 2022}.
Unfortunately, the coagulation kernel $\Lambda_{i,j} = i+j$ is not included in either of the classes \eqref{COAGRATE1} or \eqref{COAGRATE2}. Although the existence result for this kernel is available in the literature, see \cite{MPA 2022, DAS 2021, SRD 2022}.﻿
\end{remark}

We fix  $ n \geq 3$, consider the following truncated system of $n$ ordinary differential equations,
\begin{align}
\frac{d\omega_i^n(t)}{dt}&= \omega_{i-1}^n(t) \sum_{j=1}^{i-1} j \Lambda_{i-1,j} \omega_j^n(t)-\omega_i^n(t)  \sum_{j=1}^{i} j \Lambda_{i,j} \omega_j^n(t)-\sum_{j=i}^{n} \Lambda_{i,j} \omega_i^n(t) \omega_j^n(t), \hspace{.2cm} i \in \mathbb{N},\label{TDOHSE}\\
\omega_i^n(0) &= \omega_i^{in} \geq 0, \hspace{.2cm} i \in \mathbb{N}. \label{TIC}
\end{align}
The system \eqref{TDOHSE}--\eqref{TIC} has a unique solution $\omega^n =(\omega_i^n)_{1\leq i \leq n}$ defined on $[0,+\infty)$, satisfying $\omega_i^n \geq 0$ for $i =1,\cdots,n.$ The following technical lemma will be helpful in the upcoming calculations.

\begin{lemma} 
Let $(\psi_i)_{1 \leq i \leq n}$ be non-negative real numbers. If $t_2 \in[0,+\infty)$ and $t_1 \in [0,t_2]$, then 
\begin{align}
\sum_{i=1}^n \psi_i \omega_i^n(t_2) - \sum_{i=1}^{n} \psi_i \omega_i^n(t_1) =&\int_{t_1}^{t_2} \sum_{i=1}^{n-1}\sum_{j=1}^i j(\psi_{i+1}- \psi_i)\Lambda_{i,j} \omega_i^n(s) \omega_j^n(s) ds-\int_{t_1}^{t_2} j \psi_n\Lambda_{n,j} \omega_i^n(s) \omega_j^n(s) ds\nonumber\\
&-\int_{t_1}^{t_2} \sum_{i=1}^n \sum_{j=i}^n \psi_i\Lambda_{i,j} \omega_i^n(s) \omega_j^n(s) ds \label{ME}
\end{align}
\end{lemma}
%As a consequence of \eqref{ME}, we deduce the following estimates.
We deduce the following estimates as a consequence of \eqref{ME}.
\begin{lemma}
Suppose that \eqref{COAGRATE} holds. If $r \in \{1,\cdots,n\}$, $t_2\in[0,+\infty)$ and $t_1\in[0,t_2]$, then
\begin{align}
\sum_{i=1}^n i \omega_i^n(t_2) \leq \sum_{i=1}^n i \omega_i^n(t_1) \leq \sum_{i=1}^n i \omega_i^{in},\label{EST1}
\end{align}
\begin{align}
\sum_{i=1}^n \omega_i^n(t_2) + \frac{1}{2}\int_{t_1}^{t_2}  \Big|\sum_{i=1}^n \theta_i \omega_i^n(s) \Big|^2 ds \leq \sum_{i=1}^n \omega_i^{in},\label{EST2}
\end{align}
\begin{align}
\int_{t_1}^{t_2}  \Big|\sum_{i=r}^n \theta_i \omega_i^n(s) \Big|^2 ds  \leq 2 \Big(\sum_{i=1}^n i^{\eta}\omega_i^n(t_1)\Big)r^{-\eta}, \hspace{.2cm}\text{for some} \hspace{.2cm} 0<\eta <1.\label{EST3}
\end{align}
\end{lemma}
\begin{proof}

Taking $\psi_i = i $ if  $i \in \{1,\cdots,n\}$ in \eqref{ME}, we get,
\begin{align}
\sum_{i=1}^n i \omega_i(t_2)- \sum_{i=1}^{n} i \omega_i^n(t_1)  \leq 0 \label{EST1-1}
\end{align}
On replacing $t_2$ with $t_1$ and $t_1$ with $0$ in \eqref{EST1-1}, we get \eqref{EST1}.

Next let $\psi_i=1$ if $i \in \{1,\cdots,n\}$ and $\tau=0$, we have 
\begin{align*}
\sum_{i=1}^n \omega_i^n(t) -\sum_{i=1}^n \omega_i^n(\tau)  \leq -\int_{t_1}^{t_2} \sum_{i=1}^n \sum_{j=i}^n\Lambda_{i,j} \omega_i^n(s) \omega_j^n(s) ds =-\frac{1}{2}\int_{t_1}^{t_2} \sum_{i=1}^n \sum_{j=1}^n\Lambda_{i,j} \omega_i^n(s) \omega_j^n(s) ds 
\end{align*} 
Now from \eqref{COAGRATE2}, using $\theta_i \theta_j \leq \Lambda_{i,j}$, it follows that 
\begin{align*}
\sum_{i=1}^n \omega_i^n(t) +\frac{1}{2}\int_{t_1}^{t_2} \sum_{i=1}^n \sum_{j=1}^n \theta _i \theta_j \omega_i^n(s) \omega_j^n(s) ds \leq \sum_{i=1}^n \omega_i^{in}, 
\end{align*}
which gives \eqref{EST2}.

 Finally, for $(i,j) \in \{1,\cdots,n\}^2$, we take $\psi_i=1 $ if $i\geq j$ and $\psi_i=i^{\eta}$ if $i<j$ for some $0<\eta <1$ in \eqref{ME}, we have 
 \begin{align}
 \int_{t_1}^{t_2} \sum_{i=1}^n \sum_{j=i}^n i^{\eta} \Lambda_{i,j} \omega_i^n(s) \omega_j^n(s) ds \leq \sum_{i=1}^n \psi_i \omega_i^n(t_1) \leq \sum_{i=1}^n i^{\eta}\omega_i^n(t_1) \label{EST3-1}
 \end{align}

Next, we notice that
\begin{align*}
\sum_{i=1}^n \sum_{j=i}^n i^{\eta}\Lambda_{i,j} \omega_i^n \omega_j^n &= \frac{1}{2}\sum_{i=1}^n \sum_{j=i}^n i^{\eta} \Lambda_{i,j} \omega_i^n \omega_j^n + \frac{1}{2}\sum_{j=1}^n \sum_{i=j}^n j^{\eta}\Lambda_{i,j} \omega_i^n \omega_j^n\\ 
& = \frac{1}{2}\sum_{i=1}^n \sum_{j=i}^n i^{\eta}\Lambda_{i,j} \omega_i^n \omega_j^n + \frac{1}{2}\sum_{i=1}^n \sum_{j=1}^{i} j^{\eta}\Lambda_{i,j} \omega_i^n \omega_j^n\\
&= \frac{1}{2}\sum_{i=1}^n \sum_{j=1}^{n} \min\{i,j\}^{\eta}\Lambda_{i,j} \omega_i^n \omega_j^n.
\end{align*}
With the help of the above estimate, equation \eqref{EST3-1} can be rewritten as
\begin{align}
\int_{t_1}^{t_2} \sum_{i=1}^n \sum_{j=1}^{n} \min\{i,j\}^{\eta}\Lambda_{i,j} \omega_i^n \omega_j^n \leq 2 \sum_{i=1}^n i^{\eta} \omega_i^n(t_1).\label{EST3_2}
\end{align}
Finally, gathering estimates from \eqref{COAGRATE2} and \eqref{EST3_2}, we obtain

\begin{align*}
r^{\eta} &\int_{t_1}^{t_2} \sum_{i=r}^n \sum_{j=r}^n \theta_i\theta_j \omega_i^n(s) \omega_j^n(s) ds\\
&\leq \int_{t_1}^{t_2} \sum_{i=1}^n \sum_{j=1}^{n} \min\{i,j\}^{\eta} \theta_i \theta_j \omega_i^n(s) \omega_j^n(s) ds\\
&\leq \int_{t_1}^{t_2} \sum_{i=1}^n \sum_{j=1}^{n} \min\{i,j\}^{\eta}\Lambda_{i,j} \omega_i^n(s) \omega_j^n(s) ds\\
&  \leq 2 \sum_{i=1}^n i^{\eta} \omega_i^n(t_1).
\end{align*}
Therefore, \eqref{EST3} follows.

\end{proof}

Another outcome of \eqref{ME} is the following lemma, which will help to prove Theorem \ref{TH2}.

\begin{lemma}
Assume that \eqref{COAGRATE} and \eqref{COAGRATE2} hold, and let $t\in (0,+\infty)$. Then 
\begin{align}
\sum_{i=1}^n i \omega_i^n(t) \leq \frac{2}{B} \Big( \sum_{i=1}^n \omega_i^{in} \Big)^{\frac{1}{2}} t^{-\frac{1}{2}}. \label{MASSRBND}
\end{align}
\end{lemma}
\begin{proof}
 Using \eqref{COAGRATE2}, we obtain that
\begin{align*}
\theta_i \geq  Bi.
\end{align*}
Next, we conclude from \eqref{EST2} that
\begin{align}
\int_0^t \Big|\sum_{i=1}^n i \omega_i^n(s)\Big|^2 ds \leq \frac{2}{B^2} \sum_{i=1}^n \omega_i^{in}. \label{MASSRBND1}
\end{align}
As a consequence, the function
\begin{align*}
s\longrightarrow\sum_{i=1}^n i \omega_i^n(s) 
\end{align*}
is non-increasing that follows from \eqref{EST1}, hence using \eqref{MASSRBND1}, we obtain
\begin{align*}
t \Big|\sum_{i=1}^n i \omega_i^n(t) \Big|^2 \leq\frac{2}{B^2}  \sum_{i=1}^n \omega_i^{in};
\end{align*}
which shows that \eqref{MASSRBND} holds.
\end{proof}
Also, we recall a part of the well-known Rellich-Kondrachov theorem \cite{SK 2019}, to prove the compactness of the sequence of solutions to \eqref{TDOHSE}--\eqref{TIC}, which is as follows.
 \begin{theorem} \label{RKT}
 Suppose $\Omega \in \mathbb{R}^n $ is a bounded open set and the boundary of $\Omega$ is a $(n-1)$ dimensional $C^1$ manifold and if $p>n$ then $W^{1,p}$ is compactly embedded in $\mathcal{C}(\bar{\Omega})$.
 \end{theorem}
 
At the end of this section, we will provide the well-posedness result of \eqref{TDOHSE}--\eqref{TIC}.
\begin{lemma}
For each $n\geq 3$, the system \eqref{TDOHSE}--\eqref{TIC} has a unique non-negative solution $\omega^n= (\omega_i^n)_{1\leq i\leq n}$ in $\mathcal{C}^1([0,+\infty), \mathbb{R}^n))$. Moreover, we have
\begin{align*}
\sum_{i=1}^n i\omega_i^n(t) \leq \sum_{i=1}^{n} i w_i^{in}, \hspace{.2cm} t \in[0,+\infty).
\end{align*}
\end{lemma}
\begin{proof}
The proof follows exactly the same line as given in \cite[Lemma 13]{BAG 2005}.
\end{proof}

\section{\textbf{Proof of Theorem \ref{TH1}}} \label{PTH1}

%The proof of Theorem \ref{TH1} under the assumption \eqref{COAGRATE1} on the coagulation rates follows the same lines as that of \cite[Proposition 4]{BAG 2005}, and we thus omit it here. 
In this section we shall prove existence of solutions in $Y_1^+$ of \eqref{DOHSE}--\eqref{IC} with initial data in $Y_1^+$, having conditions \eqref{COAGRATE1} and \eqref{COAGRATE2} on the coagulation coefficients.

We omit the proof of Theorem \ref{TH1} with the assumption \ref{COAGRATE1} on coagulation rates because it follows the same lines as that of \cite[Proposition 4]{BAG 2005}.

%\textcolor{blue}{Since the proof of Theorem \ref{TH1} having the assumption \ref{COAGRATE1} on coagulation rates follows the same lines as that of \cite[Proposition 4]{BAG 2005}, we skip it.}

%It follows from \eqref{COAGRATE2} that, for each $i \geq 1$ and $n> i$, we have
%Next under the assumption \ref{COAGRATE2}, for each $i \geq 1$ and $n>i$, we get
%\begin{align}
%\sum_{j=q_1}^{q_2} \Lambda_{i,j} \omega_j^n \leq (1+A) \theta_i \sum_{j=q_1}^{q_2} \theta_j \omega_j^n, %\hspace{.3cm} i\leq q_1< q_2\leq n. \label{TH1PE1}
%\end{align}
Let $ i\geq 1$, for $n >i$ and $T\in(0,+\infty)$,  from  \eqref{EST2} we have, 
\begin{align}
\omega_i^n(t) \leq \|\omega^{in}\|_0, \hspace{.2cm} t \in [0,T], \label{BD1}
\end{align}
Following that, we prove the following lemma, which establishes the boundedness of $(\frac{d\omega_i^n}{dt})_{n>i}$ in $L^2(0,T)$.
\begin{lemma} \label{BD2}
Let $i \geq 1$ and $T\in(0,+\infty)$. There exists a constant $\Pi_i(T)$, depending on $\sum_{i=1}^{\infty} \omega_i^{in}$, $T$ and $i$ such that, for each $n>i$ 
\begin{align}
\Bigg|\frac{d\omega_i^n}{dt} \Bigg|_{L^2(0,T)} \leq \Pi_i(T), \hspace{.4cm} t\in [0,T]. \label{PIBND}
\end{align}
\end{lemma}

\begin{proof}
Due to \eqref{COAGRATE2},  for each $i \geq 1$ and $n>i$, we get
\begin{align}
\sum_{j=r_1}^{r_2} \Lambda_{i,j} \omega_j^n \leq (1+A) \theta_i \sum_{j=r_1}^{r_2} \theta_j \omega_j^n, \hspace{.3cm} i\leq r_1< r_2\leq n. \label{TH1PE1}
\end{align}
Now  \eqref{TDOHSE}, \eqref{EST2} and \eqref{TH1PE1} gives us

\begin{align*}
\Big| \frac{d\omega_i^n}{dt}\Big| &\leq \Big(\sum_{j=1}^{i-1}\Lambda_{i-1,j} \Big)\Big|\sum_{j=1}^n\omega_j^n\Big|^2 +\Big(\sum_{j=1}^{i}\Lambda_{i,j} \Big)\Big|\sum_{j=1}^n\omega_j^n\Big|^2  +(1+A)\theta_i \omega_i^n \sum_{j=i}^{n}  \theta_j\omega_j^n,\\
& \leq \Big(\sum_{j=1}^{i-1}\Lambda_{i-1,j} \Big) \|\omega^{in}\|_0^2+ \Big(\sum_{j=1}^{i}\Lambda_{i,j} \Big) \|\omega^{in}\|_0^2 +(1+A)\theta_i\|\omega^{in} \|_0\sum_{j=1}^n \theta_j\omega_j^n,
\end{align*}

\begin{align*}
\Big| \frac{d\omega_i^n}{dt}\Big|_{L^2(0,T)}\leq T^{\frac{1}{2}} \Big(\sum_{j=1}^{i-1}\Lambda_{i-1,j} +\sum_{j=1}^i \Lambda_{i,j} \Big) \|\omega^{in}\|_0^2 + 2(1+A) \theta_i \|\omega^{in} \|_0^{\frac{3}{2}} 
\end{align*}
Hence, the proof of lemma \ref{BD2} is completed.
\end{proof}
%Consequently, for each $i\geq 1$ and $T\in(0,+\infty)$, the sequence $(\omega_i^n)_{n>i}$ is bounded
%in $W^{1,2}(0, T)$.
As a consequence of \eqref{BD1} and Lemma \eqref{BD2}, for each $i\geq 1$ and $T\in(0,+\infty)$, the sequence $(\omega_i^n)_{n>i}$ is bounded in $W^{1,2}(0, T)$. Since it follows from Theorem \ref{RKT} that $W^{1,2}(0, T)$ is compactly embedded in $\mathcal{C}([0,T])$. Thus, $(\omega_i^n)_{n>i}$ is relatively compact in $\mathcal{C}([0,T])$ for each $i\geq 1$ and $T\in(0,+\infty)$. With the help of diagonal process we can identify a subsequence $(\omega_i^{n_l})$ of $(\omega_i^n)_{n>i}$ and a sequence $\omega=(\omega_i)_{i\geq 1}$  of non-negative continuous function such that, for each $i\geq 1$ and $T\in(0,+\infty)$ 
%Using a diagonal process, we may find a subsequence $(\omega_i^{n_l})$ of $(\omega_i^n)_{n>i}$ and a sequence $\omega=(\omega_i)_{i\geq 1}$ of non-negative continuous function such that, for each $i\geq 1$ and $T\in(0,+\infty)$ 
\begin{align}
\lim_{l \to \infty} \big|\omega_i^{n_l} -\omega_i \big|_{\mathcal{C}([0,T])} =0.\label{LIMIT}
\end{align}

  %First we show that $\omega(t) \in Y_1^+$ for each $t \in[0,+\infty)$. Indeed, let $t\in[0,+\infty)$ and $r\geq 1$. It follows from \eqref{EST1} that, for $n_l\geq r$,
  
 Let $t\in[0,+\infty)$ and $ r\geq 1$, then for $n_l\geq r$, it can be deduce from \eqref{EST1} that
\begin{align*}
\sum_{i=1}^r i \omega_i^{n_l}(t) \leq \|\omega^{in}\|_1.
\end{align*}

Next, using \eqref{LIMIT} we allow to pass to the limit $l\to +\infty$ in the previous inequality, to obtain
%We then let $l\to +\infty$, and use \eqref{LIMIT} to obtain
\begin{align*}
\sum_{i=1}^r i \omega_i(t) \leq \|\omega^{in}\|_1.
\end{align*}

%The above estimate being valid for any $r\geq 1$, by making $r\to +\infty$, we obtain
Since the above estimate holds for any $r\geq 1$, letting $r\to +\infty$, we get
\begin{align*}
\sum_{i=1}^{\infty} i \omega_i(t) \leq \|\omega^{in}\|_1.
\end{align*}
which leads to the conclusion that $\omega(t)\in Y_1^+$ and satisfies \eqref{AMC}.

Thus, we may conclude from \eqref{EST1}, \eqref{EST3}, and \eqref{TH1PE1} that, for $i \geq 1$, $T\in(0,+\infty)$ and $l\geq 1$,
\begin{align}
\int_0^T \Big| \sum_{j=r_1}^{r_2} \Lambda_{i,j} \omega_j^{n_l}(s)\Big|^2 ds \leq 4 (1+A)^2 \theta_i^2 \|\omega^{in}\|_1 r_1^{-\eta}, \hspace{.2cm} i \leq r_1<r_2\leq n_l \label{TH1PE2}
\end{align}
%Letting $ l \to +\infty$ in \eqref{TH1PE2} yields, thanks to \eqref{LIMIT},
Now thanks to \eqref{LIMIT}, we may pass to the limit as $ l \to +\infty$ in \eqref{TH1PE2} to obtain
\begin{align}
\int_0^T \Big|\sum_{j=r_1}^{r_2}\Lambda_{i,j} \omega_j(s) \Big|^2 ds \leq 4 (1+A)^2 \theta_i^2 \|\omega^{in}\|_1 r_1^{-\eta}, \hspace{.2cm} i \leq r_1<r_2. \label{TH1PE3}
\end{align}
%Since the right hand side of \eqref{TH1PE3} does not depend on $r_2$, a first consequence of \eqref{TH1PE3} is that 
Because the right-hand side of \eqref{TH1PE3} does not depend on $r_2$, the first conclusion of \eqref{TH1PE3} is that for each $i\geq 1$ and $T\in(0,+\infty)$, $\sum_{j=i}^{\infty} \Lambda_{i,j} \omega_j$ belongs to $L^2(0,T)$, and
\begin{align}
\int_0^T \Big|\sum_{j=r_1}^{\infty}\Lambda_{i,j} \omega_j(s) \Big|^2 ds \leq 4 (1+A)^2 \theta_i^2 \|\omega^{in}\|_1 r_1^{-\eta}, \hspace{.2cm} r_1 \geq i. \label{TH1PE4}
\end{align}
%We now prove that, for $ i\geq 1$ and $T\in(0,+\infty)$,
%\begin{align}
%\lim_{l\to\infty} \Big|\sum_{j=i}^{n_l} \Lambda_{i,j} \omega_j^{n_l} -\sum_{j=i}^{\infty} \Lambda_{i,j} \omega_j\Big|_{L^2(0,T)} =0. \label{TH1PE5}
%\end{align}
In order to pass to the limit in \eqref{TDOHSE}, we shall show that for each $i\geq 1$,
\begin{align}
\sum_{j=i}^{n_l}\Lambda_{i,j} \omega_j^{n_l} \longrightarrow \sum_{j=i}^{\infty} \Lambda_{i,j} \omega_j \hspace{.2cm} \text{in} \hspace{.2cm}L^2(0,T), \label{TH1LIM3}
\end{align}
as $l\to \infty$.

Indeed, for each $r>i$ and $n_l\geq r$, we infer from \eqref{TH1PE2} and \eqref{TH1PE4} that
\begin{align*}
\Bigg|\sum_{j=i}^{n_l} \Lambda_{i,j}&\omega_j^{n_l} - \sum_{j=i}^{\infty} \Lambda_{i,j} \omega_j\Bigg|_{L^2(0,T)}\\
& \leq \sum_{j=i}^{r-1} \Lambda_{i,j}|\omega_j^{n_l} - \omega_j|_{L^2(0,T)} + \Bigg| \sum_{j=r}^{n_l} \Lambda_{i,j} \omega_j^{n_l} \Bigg|_{L^2(0,T)} + \Bigg|\sum_{j=r}^{\infty} \Lambda_{i,j} \omega_j \Bigg|_{L^2(0,T)}\\
& \leq \sum_{j=i}^{r-1}\Lambda_{i,j} |\omega_j^{n_l} - \omega_j|_{L^2(0,T)}  + 4 (1+A)\theta_i \|\omega^{in}\|_1^{\frac{1}{2}} r^{-\frac{\eta}{2}}.
\end{align*}
%Owning to \eqref{LIMIT}, we may pass to the limit as $l\to +\infty$ in the above estimate, and obtain 
Owing to \eqref{LIMIT}, we may pass to the limit as $l\to +\infty$ in the preceding estimate and obtain

\begin{align}
\limsup_{l \to +\infty} \Bigg| \sum_{j=i}^{n_l} \Lambda_{i,j} \omega_j^{n_l} - \sum_{j=i}^{\infty} \Lambda_{i,j} \omega_j \Big|_{L^2(0,T)} \leq 4 (1+A)\theta_i \|\omega^{in}\|_1^{\frac{1}{2}} r^{-\frac{\eta}{2}}
\end{align}

for any $ q >i$; hence \eqref{TH1LIM3} follows.\par
 %For each $i\geq 2$, using \eqref{BD1}, we have 
%\begin{align*}
%\Bigg| \sum_{j=1}^{i-1} \Lambda_{i-1,j} \omega_{i-1}^{n_l}\omega_j^{n_l}-  \sum_{j=1}^{i-1} \Lambda_{i-1,j} \omega_{i-1}\omega_j \Bigg|_{L^2(0,T)} &\leq \|\omega^{in}\|_0\sum_{j=1}^{i-1} \Lambda_{i-1,j}\big|\omega_{i-1}^{n_l} - \omega_{i-1}\big|_{L^2(0,T)} \\
%&+\|\omega^{in}\|_0\sum_{j=1}^{i-1} \Lambda_{i-1,j}\big|\omega_{j}^{n_l} - \omega_{j}\big|_{L^2(0,T)}. 
%\end{align*}
%Since for each fixed $i\geq 2$, right hand side of above inequality contains only finitely many terms. Hence owning to \eqref{LIMIT}, we may pass to the limit as $l\to \infty$, we obtain
%\begin{align}
%\limsup_{l\to\infty} \Bigg| \sum_{j=1}^{i-1} j\Lambda_{i-1,j} \omega_{i-1}^{n_l}\omega_j^{n_l}-  \sum_{j=1}^{i-1}j \Lambda_{i-1,j} \omega_{i-1}\omega_j \Bigg|_{L^2(0,T)} =0
%\end{align}
Now thanks to \eqref{LIMIT} and \eqref{BD1}, it is easy to show that for $i\geq 1$, we have 
\begin{align}
\lim_{l\to\infty} \Bigg| \sum_{j=1}^{i-1} j\Lambda_{i-1,j} \omega_{i-1}^{n_l}\omega_j^{n_l}-  \sum_{j=1}^{i-1}j \Lambda_{i-1,j} \omega_{i-1}\omega_j \Bigg|_{L^2(0,T)} =0,\label{TH1LIM1}
\end{align}
and
\begin{align}
\lim_{l\to\infty} \Bigg| \sum_{j=1}^{i} j\Lambda_{i,j} \omega_{i}^{n_l}\omega_j^{n_l}-  \sum_{j=1}^{i}j \Lambda_{i,j} \omega_{i}\omega_j \Bigg|_{L^2(0,T)} =0. \label{TH1LIM2}
\end{align}

%Finally with the help of \eqref{TH1LIM1},\eqref{TH1LIM2} and \eqref{TH1LIM3}, it is straightforward to pass to the limit in the $i$-th equation of \eqref{TDOHSE} as $l \to +\infty$, and conclude that $\omega$ is a solution to \eqref{DOHSE}--\eqref{IC} on $[0,+\infty)$. The proof of Theorem \ref{TH1} is thus complete.

Finally, using \eqref{TH1LIM3}, \eqref{TH1LIM1}, and \eqref{TH1LIM2}, it is straightforward to pass to the limit in the $i$-th equation of \eqref{TDOHSE} as $l \to +\infty$, and infer that $\omega$ is a solution to \eqref{DOHSE}—\eqref{IC} on $[0,+\infty)$. Thus, the proof of Theorem \ref{TH1} is complete.

\section{\textbf{Proof of Theorem \ref{TH2}}} \label{PTH2}

%Since the bounds \eqref{BD1} and \eqref{BD2} only depend on the $Y_0$-norm of $\omega^{in}$, the same argument as in the proof of Theorem \ref{TH1} yields the existence of a subsequence $(\omega_i^{n_l})$ of $(\omega_i^n)_{n >i}$, and a sequence $\omega=(\omega_i)_{i\geq 1}$ of non-negative continuous functions, such that, for each $i \geq 1$ and $T \in (0,+\infty)$,

Owing to the fact that the bounds \eqref{BD1}1 and \eqref{BD2} depend only on the $Y 0$-norm of $\omega^{in}$, the same argument that is used to prove Theorem \ref{TH1} leads to the conclusion that there exists a subsequence $(\omega_i^{n_l})$ of $(\omega_i^n)_{n >i}$, and a sequence $\omega=(\omega_i)_{i\geq 1}$ of non-negative continuous functions, such that, for each $i \geq 1$ and $T \in (0,+\infty)$,

\begin{align}
\lim_{l \to \infty} \big|\omega_i^{n_l} -\omega_i \big|_{\mathcal{C}([0,T])} =0.\label{LIMIT1}
\end{align}

Let $t \in(0,+\infty)$ and $r\geq i$. If $n_l\geq r$, we can deduce from \eqref{EST2} and \eqref{EST3} that 
\begin{align*}
\sum_{i=1}^r \omega_i^{n_l}(t) \leq \|\omega^{in}\|_0 \hspace{.2cm} \text{and} \hspace{.2cm} \sum_{i=1}^r i \omega_i^{n_l}(t) \leq \frac{2}{B}\|\omega^{in}\|_0^{\frac{1}{2}} t^{-\frac{1}{2}}.
\end{align*}

With the help of \eqref{LIMIT1}, we let $l \to +\infty$ in the previous estimates to obtain
\begin{align*}
\sum_{i=1}^r \omega_i(t) \leq \|\omega^{in}\|_0 \hspace{.2cm} \text{and} \hspace{.2cm} \sum_{i=1}^r i \omega_i(t) \leq \frac{2}{B}\|\omega^{in}\|_0^{\frac{1}{2}} t^{-\frac{1}{2}}.
\end{align*} 

%The above estimates being valid for any $ r\geq 1$, we have in fact
Since the aftermentioned estimate holds for any $r\geq 1$, we have
\begin{align}
\|\omega(t)\|_0 \leq \|\omega^{in}\|_0 \hspace{.2cm} \text{and} \hspace{.2cm} \|\omega(t)\|_1 \leq \frac{2}{B}\|\omega^{in}\|_0^{\frac{1}{2}} t^{-\frac{1}{2}}; \label{TH2PE1}
\end{align}
hence \eqref{FM} follows.

Next, let $i \geq 1$ and $T\in(0,+\infty)$. For $r\geq i$ and $n_l\geq r$, \eqref{EST2} and \eqref{TH1PE1}  yields
\begin{align*}
\int_0^T \Bigg|\sum_{j=i}^r \Lambda_{i,j} \omega_j^{n_l}(s) \Bigg|^2 ds \leq 2 (1+A)^2 \theta_i^2 \|\omega^{in}\|_0
\end{align*}
%We first pass to the limit as $l\to +\infty$ using \eqref{LIMIT1}, and then let $r\to +\infty$, to conclude that
Using \eqref{LIMIT1}, we first pass to the limit as $l\to +\infty$, and then let $r\to +\infty$ to argue that
\begin{align}
\sum_{j=i}^{\infty} \Lambda_{i,j} \omega_j \in L^2(0,T).\label{TH2PE2}
\end{align}

%We next consider $t_1 >0$ and $T\in(t_1,+\infty)$. Using \eqref{TH2PE1}, it is clear that $\omega(t_1)\in Y_1^+$, hence, it follows from \eqref{EST3}, \eqref{TH1PE1} and \eqref{LIMIT1} by a proof similar to that of \eqref{TH1PE2}--\eqref{TH1LIM3} that, for each $i \geq 1$, we have 

We next consider $t_1 >0$ and $T\in(t_1,+\infty)$. It is clear from \eqref{TH2PE1} that $\omega(t_1) \in Y_1+$. Next, following the calculations that we performed to get \eqref{TH1PE2}--\eqref{TH1LIM3}, we can now use \eqref{EST3}, \eqref{TH1PE1}, and \eqref{LIMIT}, and for each $i \geq 1$, to obtain

%Now using \eqref{EST3}, \eqref{TH1PE1} and \eqref{LIMIT1} and following the calculation which we have done to obtained \eqref{TH1PE2}--\eqref{TH1LIM3}, for each $i \geq 1$, we get

\begin{align}
\lim_{l\to\infty} \Big|\sum_{j=i}^{n_l} \Lambda_{i,j} \omega_j^{n_l} -\sum_{j=i}^{\infty} \Lambda_{i,j} \omega_j\Big|_{L^2(0,T)} =0. \label{TH2PE3}
\end{align}
We are now ready to verify that $\omega$ is a solution to \eqref{DOHSE}--\eqref{IC} on $[0,+\infty)$. Let $i=1$ and $n_l\geq i$. Using \eqref{TDOHSE}, we have
%We are now ready to check that $\omega$ is a solution to \eqref{DOHSE}-\eqref{IC} on $[0,+\infty)$. Let $i \geq 1$ and $n_l\geq i$. Recalling \eqref{TDOHSE}, we have

 \begin{align}
 \omega_i^{n_l}(t_2)= \omega_i^{n_l}(t_1) + \int_{t_1}^{t_2} \Big( \omega_{i-1}(s) \sum_{j=1}^{i-1} j \Lambda_{i-1,j} \omega_j(s)-\omega_i(s)  \sum_{j=1}^{i} j \Lambda_{i,j} \omega_j(s)-\sum_{j=i}^{\infty} \Lambda_{i,j} \omega_i(s) \omega_j(s)\Big) ds, \label{TH2PE4}
 \end{align}
for each $t_2 \in (0,+\infty)$ and $t_1\in(0,t_2)$. Now thanks to \eqref{LIMIT1} and \eqref{TH2PE3}, we may pass to the limit as $l\to +\infty$, and obtain

\begin{align*}
 \omega_i(t_2)= \omega_i(t_1) + \int_{t_1}^{t_2} \Big( \omega_{i-1}(s) \sum_{j=1}^{i-1} j \Lambda_{i-1,j} \omega_j(s)-\omega_i(s)  \sum_{j=1}^{i} j \Lambda_{i,j} \omega_j(s)-  \omega_i(s) \sum_{j=i}^{\infty} \Lambda_{i,j} \omega_i(s) \omega_j(s)\Big) ds
 \end{align*}
for each $t_2\in(0, +\infty)$ and $t_1 \in(0,t_2)$ .\\
%Since $\omega\in \mathcal{C}([0,t])$ and $\sum_{j=1}^{\infty} \Lambda_{i,j} \omega_j$ belongs to $L^1(0,t)$ by \eqref{TH2PE2}, we may let $\tau \to 0$ in the previous identity and conclude that $\omega_i$ satisfies the $i$-th equation of \eqref{DOHSE} in the sense of definition given in \eqref{DEF1}. Therefore $\omega$ is a solution to \eqref{DOHSE}--\eqref{IC} on $[0,+\infty)$, and the proof of Theorem \ref{TH2} is complete. 
 Notice that as $\omega\in \mathcal{C}([0,t_2])$ and $\sum_{j=1}^{\infty} \Lambda_{i,j} \omega_j$ belongs to $L^1(0,t_2)$ by \eqref{TH2PE2} we may let $t_1 \to 0$ and deduce that $\omega_i$ satisfies the $i$-th equation of \eqref{DOHSE} in the sense of definition given in \eqref{DEF1}. As a result, $\omega$ is a solution to \eqref{DOHSE}--\eqref{IC} on $[0,+\infty)$, completing the proof of Theorem \ref{TH1}.

\section{\textbf{Gelation Phenomenon in equation \eqref{DOHSE}--\eqref{IC}}} \label{GP}

%%%%%%%%%%%%%%%%%%%%%%%%%%%%%%%%%%%%
%%%%%%%%%%%%%%%%%%%%%%%%%%%%%%%%%%%%

%%%%%%%%%%%%%%%%%%%%%%%%%%%%%%%%%%%%
%%%%%%%%%%%%%%%%%%%%%%%%%%%%%%%%%%%%
 %%%%%%%%%%%%%%%%%%%%%%%%%%%%%%%%%%%%
%%%%%%%%%%%%%%%%%%%%%%%%%%%%%%%%%%%%

%%%%%%%%%%%%%%%%%%%%%%%%%%%%%%%%%%%%
%%%%%%%%%%%%%%%%%%%%%%%%%%%%%%%%%%%%
 %%%%%%%%%%%%%%%%%%%%%%%%%%%%%%%%%%%%
%%%%%%%%%%%%%%%%%%%%%%%%%%%%%%%%%%%%

%%%%%%%%%%%%%%%%%%%%%%%%%%%%%%%%%%%%
%%%%%%%%%%%%%%%%%%%%%%%%%%%%%%%%%%%%
 \begin{theorem}
Assume that $\Lambda_{i,j}$ satisfies \eqref{COAGRATE} and 
\begin{align}
   \Lambda_{i,j} \geq C (ij)^{\frac{\kappa_0}{2}}, \hspace{.3cm}  (i,j) \in \mathbb{N}^2,\label{LBGamma}
   \end{align}
 for some  $ \kappa_0 \in (1,2)$ and $C>0$. Consider $ \omega^{in} \in Y_1^+,$ $ \omega^{in}\not \equiv 0$, and denote $\omega = (\omega_i)_{i\geq 1}$ a solution to \eqref{DOHSE}--\eqref{IC}. Then $T_{\text{gel}} < +\infty$.
\end{theorem}
\begin{proof}
 % For $ t_2 \geq t_1 \geq 0$, multiplying \eqref{DOHSE} with $\phi_i$ and taking sum over $i$, after a simple calculation, we have 
 For $ t_2 \geq t_1 \geq 0 $, we multiply \eqref{DOHSE} by $ \phi_i $ and take the sum over $i$ from $1$ to $q$. After a simple calculation, we have
  \begin{align}
  \sum_{i=1}^r \phi_{i}\omega_i(t_2) -\sum_{i=1}^r \phi_{i}\omega_i(t_1)=&\int_{t_1}^{t_2}\Bigg[\sum_{i=1}^{r-1}\sum_{j=1}^{i}j \phi_{i+1}\Lambda_{i,j}\omega_{i}(s)\omega_{j}(s) -\sum_{i=1}^{r}\sum_{j=1}^{i}(j\phi_{i}+\phi_{j})\Lambda_{i,j}\omega_{i}(s) \omega_{j}(s)\nonumber\\
   &-\sum_{i=r+1}^{\infty}\sum_{j=1}^{r}\phi_{j}\Lambda_{i,j}\omega_{i}(s) \omega_{j}(s)\Bigg] ds.\label{GEL1}
\end{align}

 As $ \omega= (\omega_i)_{i\geq 1} \in Y_1^+$ given by Theorem \ref{TH1}, hence it follows from above inequality that for $\phi_i=i$, the function $t \longrightarrow \sum_{i=1}^{\infty} i \omega_i(t)$ is non-increasing, i.e.,
 \begin{align}
 \sum_{i=1}^{\infty} i \omega_i(t_2) \leq \sum_{i=1}^{\infty} i \omega_i(t_1). \label{MDECREASE}
 \end{align}

 Now we consider another sequence $\Phi = (\Phi_i)_{i \geq 1}$ decaying sufficiently rapidly and each solution $ \omega= (\omega_i)_{i\geq 1}$ of \eqref{DOHSE}-- \eqref{IC}, we have,
  %Next, we notice that equation \eqref{ME} can be rewritten as 
  \begin{align*}
\sum_{i=1}^{\infty} \Phi_i \omega_i(t_2)-\sum_{i=1}^{\infty} \Phi_i \omega_i(t_2)= \int_{t_1}^{t_2}\sum_{i=1}^{\infty} \sum_{j=1}^{i}[j(\Phi_{i+1}- \Phi_i)-\Phi_j] \Lambda_{i,j} \omega_i(s) \omega_j(s)ds. 
\end{align*}
Let $\Phi_i= \min(i,r)$, then for $i\geq j \geq (r-1)$, we have $[j(\Phi_{i+1}- \Phi_i)-\Phi_j] \leq -r$ and  $[j(\Phi_{i+1}- \Phi_i)-\Phi_j]\leq 0$, otherwise. Since  
\begin{align*}
\sum_{i=1}^{\infty} \Phi_i \omega_i(t_2)-\sum_{i=1}^{\infty} \Phi_i \omega_i(t_1) =\int_{t_1}^{t_2} \Big(\sum_{i=1}^{r-1} \sum_{j=1}^{i} + \sum_{i=r}^{\infty}\sum_{j=1}^{r-1} + \sum_{i=r}^{\infty} \sum_{j=r}^i \Big)[j(\Phi_{i+1}- \Phi_i)-\Phi_j] \Lambda_{i,j} \omega_i(s) \omega_j(s)ds.
\end{align*}
   Hence 
 \begin{align*}
 \sum_{i=1}^{\infty} \Phi_i \omega_i(t_2)-\sum_{i=1}^{\infty}\Phi_i \omega_i(t_1) \leq \frac{-r}{2} \int_{t_1}^{t_2} \sum_{i=r}^{\infty} \sum_{j=r}^{\infty}  \Lambda_{i,j} \omega_i(s) \omega_j(s)ds.
\end{align*}   
   \begin{align}
   \int_{t_1}^{t_2} \sum_{i=r}^{\infty} \sum_{j=r}^{\infty}  \Lambda_{i,j} \omega_i(t) \omega_j(t)dt \leq \frac{2}{r}\sum_{i=1}^{\infty} i \omega_i(t_1). \label{TAILEST}
   \end{align}
Next, consider
   %$\ell_i = i^{1- \frac{B}{2}}$ and notice that 
   \begin{align*}
   \jmath = \sum_{i=1}^{\infty} i^{-\frac{(\kappa_0 +1)}{2}} < \infty
   \end{align*}
   since $ \kappa_0 +1 >2$. With the help of the Holder's inequality, \eqref{LBGamma}, and \eqref{TAILEST}, it follows for $t_2\geq t_1 \geq 0$ that
   
    %For $t_2 \geq t_1 \geq 0$, with the help of the Holder inequality, \eqref{LBGamma}, and \eqref{TAILEST}, it follows that
 
   \begin{align*}
   \int_{t_1}^{t_2}& \Big(\sum_{k=1}^{\infty} k^{-\frac{\kappa_0}{2}} \sum_{i=k}^{\infty} i^{\frac{\kappa_0}{2}} \omega_i(s) \Big)^2 ds\\
   &\leq \jmath \int_{t_1}^{t_2} \sum_{k=1}^{\infty} k^{-\frac{\kappa_0}{2}} k^{\frac{1}{2}}\Big(\sum_{i=k}^{\infty} i^{\frac{\kappa_0}{2}} \omega_i(s) \Big)^2 ds \\
   & \leq \frac{\jmath}{C} \sum_{k=1}^{\infty} k^{-\frac{\kappa_0}{2}} k^{\frac{1}{2}} \int_{t_1}^{t_2} \sum_{i=k}^{\infty} \sum_{j=k}^{\infty} \Lambda_{i,j} \omega_i(s) \omega_j(s) ds\\
   & \leq \frac{2\jmath^2 }{C}M_1(t_1).
   \end{align*}
   However 
   \begin{align*}
   \sum_{k=1}^{\infty} k^{-\frac{\kappa_0}{2}} \sum_{i=k}^{\infty}i^{\frac{\kappa_0}{2}} \omega_i(s) = \sum_{i=1}^{\infty} \Big( \sum_{k=1}^i k^{-\frac{\kappa_0}{2}} \Big)i^{\frac{\kappa_0}{2}} \omega_i(s) \geq \sum_{i=1}^{\infty} i \omega_i(s)
   \end{align*}
   Combining the previous inequalities yields
   \begin{align}
   \int_{t_1}^{t_2} \Big(\sum_{i=1}^{\infty} i \omega_i(s)\Big)^2ds \leq D M_1(t_1) \label{MEST}
   \end{align}
   for some constant $D$ depending only on $ \kappa_0$ and $C$. Next, we deduce from \eqref{MDECREASE}  and \eqref{MEST} that
   \begin{align}
   \int_0^{\infty} M_1(s)^2 ds \leq D \|\omega^{in}\|_1. \label{M1INT}
   \end{align}
   %It follows from \eqref{MDECREASE} and \eqref{M1INT} that the total mass $M_1$ is a non-increasing and non-negative function which also belongs to $L^2(0,+\infty)$. Hence 
 %  Recalling \eqref{MDECREASE}, we realize that the total mass $M_1$ is a non-increasing and non-negative function which also belongs to $L^2(0,+\infty)$. Therefore

From \eqref{MDECREASE} and \eqref{M1INT}, it can be inferred that the total mass $M_1$ is a non-increasing, non-negative function that also belongs to $L^2(0,+\infty)$. Hence

   \begin{align*}
   \lim_{t \to \infty} M_1(t) = 0
   \end{align*}
   which clearly indicates that $T_{\text{gel}} < +\infty$ since $M_1(0) > 0$.
\end{proof}
Next, we take into account the case when $\kappa_0 = 2$; in this case, we can explicitly calculate gelation time.
\begin{prop}
Assume that $\Lambda_{i,j}$ satisfy \eqref{COAGRATE} and 
\begin{align}
\Lambda_{i,j} \geq \zeta ij,  \label{ProductGrowth}
\end{align}
for $i,j \geq 1$ and where constant  $\zeta$ is a positive real number. \\
Consider $T \in (0,+ \infty)$ and  $ \omega^{in}\in Y_1^+$ and assume that \eqref{DOHSE}--\eqref{IC} has a solution $\omega \in Y_1^+ $. Then 
\begin{align}
\lim_{t\to \infty} \|\omega(t) \|_1 =0 
\end{align}
\end{prop}
\begin{proof}
Under the given assumptions \eqref{ProductGrowth}, existence of the solution is given by Theorem \ref{TH1}. Let $t_2 \geq t_1\geq 0$ and taking $\phi_i=1$ in \eqref{GEL1}, and  passing to the limit as $r\to \infty$ with the help of \eqref{EST2}, we have

\begin{align}
\sum_{i=1}^{\infty} \omega_i(t_2) - \sum_{i=1}^{\infty} \omega_i(t_1) \leq -  \frac{1}{2}\int_{t_1}^{t_2} \sum_{i=1}^{\infty}\sum_{j=1}^{\infty}\Lambda_{i,j}\omega_{i}(s) \omega_{j}(s)ds \label{EST8}
\end{align}

Using the lower bound in \eqref{ProductGrowth}, we finally arrive at

\begin{align}
\sum_{i=1}^{\infty} \omega_i(t_2) +\frac{\zeta}{2}\int_{t_1}^{t_2} \|\omega(s)\|_1^2 dt  \leq \sum_{i=1}^{\infty} \omega_i(t_1) 
\end{align}

%Now consider $t \in (0,+\infty)$. As from \eqref{MDECREASE}, we know that $t\longrightarrow \|\omega(t)\|_1$ is non-increasing, we deduce from the previous estimate (with $s=0$ and $\tau=t$) that

Now consider $t \in (0,+\infty)$. We know from \eqref{MDECREASE} that $t\longrightarrow \|\omega(t)\|_1$ is not increasing, and we infer from the foregoing estimate (with $t_1=0 $ and $t_2=t$) that

\begin{align*}
\frac{\zeta t}{2} \|\omega(t)\|_1^2 \leq \sum_{i=1}^{\infty} \omega_i^{in} \leq \sum_{i=1}^{\infty}i \omega_i^{in}= \|\omega^{in}\|_1
\end{align*}

Thus 
\begin{align*}
\| \omega(t) \|_1 \leq \Bigg( \frac{2\|\omega^{in}\|_1 }{\zeta t } \Bigg)^{\frac{1}{2}}, \hspace{.3cm}  t \in (0,+\infty),
\end{align*}
%and the proof of Proposition is complete.\\
that concludes the proof of the Proposition.\\
\end{proof}
%Recalling \eqref{TH2PE1}, we have in fact proved a stronger result than \eqref{FM}, namely
Now recall \eqref{TH2PE1}, it can be easily noticed that we have proved a much stronger result than \eqref{ME}, namely
\begin{prop}
Assume that \eqref{COAGRATE} and \eqref{COAGRATE2} hold, and that $\omega^{in} =(\omega)_{i\geq 1} \in Y_0^+$. Then there exists at least one solution $\omega$ to \eqref{DOHSE}--\eqref{IC} on $[0,+\infty)$ such that $\omega(t) \in Y_0^+$ for each $t\in[0,+\infty)$, and, for each $t>0$,
\begin{align}
\|\omega(t)\|_1 \leq \frac{2}{B} \|\omega^{in}\|_0^{\frac{1}{2}} t^{-\frac{1}{2}}.
\end{align}
\end{prop}
%The above result implies that, for coagulation rates satisfying (2.3), gelation occurs even if initially the total mass is infinite.
According to the above result, gelation happens for coagulation rates satisfying \eqref{COAGRATE2}, even when the initial total mass is infinite.

Finally, in the following section, we will demonstrate that as time increases to infinity, the number of particles decreases to zero.
\section{\textbf{Appendix}}
 %\subsection*{\textbf{Alternative Proof of Theorem \ref{thm1}}}
From \eqref{EST8}, it follows that
\begin{align*}
\sum_{i=1}^{\infty} \omega_i(t_1) + \frac{1}{2}\int_{t_1}^{t_2} \sum_{i=1}^{\infty}\sum_{j=1}^{\infty}\Lambda_{i,j}\omega_{i}(s) \omega_{j}(s)ds \leq \sum_{i=1}^{\infty} \omega_i(t_1)  
\end{align*}
which implies that
\begin{align*}
\int_{t_1}^{t_2} \Big(\sum_{i=1}^{\infty}\omega_{i}(s) \Big)^2 ds \leq \frac{2}{C}\sum_{i=1}^{\infty} \omega_i(t_1)
\end{align*}

\begin{align*}
\int_0^{\infty} \Big(\sum_{i=1}^{\infty} \omega_i(s)\Big)^2 ds \leq \sum_{i=1}^{\infty} i \omega_i^{in} 
\end{align*}

We can see from \eqref{EST8} that the total number of particles $M_0$ is a non-increasing and non-negative function of time that also belongs to $L^2(0, +\infty)$. Therefore
\begin{align*}
\lim_{t\to \infty} M_0(t) =0. 
\end{align*}

\par

%
%\subsection*{Acknowledgments}
% This work is  partially supported by Department of Science \& Technology (DST), India-Deutscher Akademischer Austauschdienst (DAAD) within the Indo-German joint project entitled ''Analysis and Numerical Methods for Population Balance Equations''.
%\subsection*{\textbf{Funding Information}}
%MA would like to thank the University Grant Commission (UGC), India for granting the Ph.D. fellowship through Grant No. 416611.
%\subsection*{\textbf{Competing Interests}}
%The authors declare that they have no conflict of interests.
%\subsection*{Author Contribution}
%All the authors have contributed equally.


\begin{thebibliography}{}
\bibitem{MPA 2022} Ali, M., Rai, P., Giri, A., K., On the discrete Safronov-Dubovskii coagulation equation: well-posedness, mass-conservation and asymptotic behaviour, \textit{arXiv:2206.10271}, 2022.




\bibitem{BAG 2005} Bagland,V., Convergence of a discrete Oort–Hulst–Safronov equation, \textit{Math. Methods Appl. Sci.}, \textbf{28(13)}, 1613--1632, 2005.


\bibitem{BP 2007} Bagland, V., Laurencot, Ph., Self-similar solutions to the Oort-Hulst-Safronov coagulation equation, \textit{SIAM J. Math. Anal.}, \textbf{39}, 345--378, 2007.

\bibitem{BLL 2019}  Banasiak, J., Lamb, W. , and  Lauren\c{c}ot, Ph.,  \textit{Analytic Methods for Coagulation-Fragmentation Models, Volume 1 \& 2}, CRC Press, Boca Raton, 2019.

\bibitem{BPA 2022} Barik, P., K., Rai, P. and  Giri, A., K., Mass-conserving weak solutions to Oort-Hulst-Safronov coagulation equation with singular rates, \textit{J. Differ. Equ.}, \textbf{11(5)}, 1125--1138, 2022.


\bibitem{HEZ 1983} Hendriks, E., M., Ernst, M., H., and Ziff, R., M., Coagulation equations with gelation, \textit{J. Stat. Phys.}, \textbf{31}, 519--563, 1983.


\bibitem{JI 1998} Jeon, I., Existence of gelling solutions for coagulation–fragmentation equations, \textit{Commun. Math. Phys.}, \textbf{194}, 541--567, 1998.






\bibitem{DAS 2021} Das, A., Saha, J., On the global solutions of discrete Safronov–Dubovski\v{i} aggregation equation, \textit{Z. Angew. Math. Phys.}, \textbf{72(183)}, 2021.

\bibitem{DAV 2014} Davidson, J., Existence and uniqueness theorem for the Safronov-Dubovski\v{i} coagulation equation, \textit{Zeitschrift f¨ur angewandte Mathematik und Physik}, \textbf{65(4)}, 757--766, 2014.



\bibitem{RLD 1972} Drake, R., L., A general mathematical survey of the coagulation equation, in Topics in
current aerosol research, part 2, International Reviews in Aerosol Physics and Chemistry, 203–376, \textit{Pergamon Press, Oxford}, 1972.



\bibitem{PBD1 1999} Dubovski\v{i}, P.B., Structural stability of disperse systems and finite nature of a coagulation front, \textit{J. Experim. Theor. Phys.}, \textbf{89(2)}, 384--390, 1999.


\bibitem{PBD 1999} Dubovski\v{i}, P.,B., A ‘triangle’ of interconnected coagulation models, \textit{J. Phys. A: Math. Gen.}, \textbf{32(5)}, 781--793, 1999.





\bibitem{FSK 2000} Friedlander, S., K., \textit{Smoke, Dust, and Haze: Fundamentals of Aerosol Dynamics}, 2nd Edition, Oxford University Press, New York (2000).

\bibitem{GL 1995} Gueron, S., Levin, S., A., The dynamics of group formation, \textit{Math. Biosciences}, \textbf{128}, 243--264, 1995.

\bibitem{SRD 2022} Kaushik, S., Kumar, R., Da Costa, F.,P., Theoretical analysis of a discrete population balance model for sum kernel, \textit{arXiv:2206.01965}, 2022.


\bibitem{SR 2022} Kaushik, S., Kumar, R., Existence, Uniqueness and Mass Conservation for Safronov-Dubovski Coagulation Equation, \textit{Acta Appl Math}, \textbf{179(10)}, 2022.



\bibitem{SK 2019} Kesavan, S., \textit{Topics in Functional Analysis and Applications}, New Age International Publishers; Third edition, 2019.


\bibitem{LLW 2003} Lachowicz, M., Lauren\c{c}ot Ph. and Wrzosek, D., On the Oort-Hulst-Safronov coagulation equation and its relation to the Smoluchowski equation, \textit{SIAM J. Math. Anal.}, \textbf{34}, 1399--1421, 2003.

\bibitem {PL 1999} Laurençot, Ph., Global solutions to the discrete coagulation equations, \textit{Mathematika}, \textbf{46(2)}, 433--442, 1999.

\bibitem{PL 2005} Lauren\c{c}ot, Ph., Convergence to self-similar solutions for a coagulation equation, \textit{Z. Angew. Math. Phys.}, \textbf{56}, 398--411, 2005.

 \bibitem{PL 2006} Lauren\c{c}ot, Ph., Self-similar solutions to a coagulation equation with multiplicative kernel, \textit{Physica D}, \textbf{222}, 80--87, 2006.



\bibitem{LT 1982} Leyvraz, F., Tschudi, H., R., Critical kinetics near gelation, \textit{J. Phys. A}, \textbf{15}, 1951--1964, 1982.





\bibitem{OH 1946} Oort, J., H., Van de Hulst, H.C., Gas and smoke in interstellar space, \textit{Bull. Astronom. Inst. Netherlands}, \textbf{10}, 187--210, 1946.

\bibitem{SMOL 1916} Smoluchowski, M., V., Drei vortrage uber diffusion, brownsche bewegung und koagulation von kolloidteilchen, \textit{ Zeitschrift f\"{u}r Physik}, \textbf{17}, 557–585, 1916.

\bibitem{SMOL 1917} Smoluchowski, M., V., Versuch einer mathematischen Theorie der Koagulationskinetik kolloider L\"{o}sungen, \textit{Z. Phys. Chem.}, \textbf{92}, 129–168, 1917.










 \end{thebibliography}
\end{document}